\documentclass{article}

\usepackage{amsmath}
\usepackage{latexsym}
\usepackage{verbatim}
\usepackage{amssymb}
\usepackage{graphicx}
\newtheorem{theorem}{Theorem}

\newtheorem{corollary}{Corollary}

\newtheorem{lemma}{Lemma}

\newtheorem{proposition}[theorem]{Proposition}

\newenvironment{proof}[1][Proof]{\noindent\textbf{#1.} }{\ \rule{0.5em}{0.5em}}

\begin{document}

\title{The Stern diatomic sequence via generalized Chebyshev polynomials}
\author{V. De Angelis}

\maketitle

The Stern diatomic sequence is defined by
\begin{equation}
\label{def} a(0)=0,\ \  a(1)=1,\ \  a(2n)=a(n),\ \  a(2n+1)=a(n)+a(n+1).\end{equation}
See \cite{Sam} for a nice survey article.

Define polynomials \(q_r(y_1,\ldots,y_r)\) inductively by
\begin{eqnarray}& q_0=1, \ q_1(y_1)=y_1, \notag \\ &q_r(y_1,\ldots,y_r)=y_1q_{r-1}(y_2,\ldots,y_r)-q_{r-2}(y_3,\ldots,y_r)
 \ \mbox{for } r\geq 2. \label{recur} \end{eqnarray}

The main result of this note is that \(a(n)\) coincides with these polynomials when the variables are given the value of the
gaps between successive 1's in the binary expansion of $n$, increased by 1 (Theorem \ref{th}), and a formula expressing
the same polynomials in terms of sets of increasing integers of alternating parity (Theorem \ref{l1}).

The polynomials \(q_r\) have appeared before as generalized Chebyshev polynomials
in the context of cluster algebras, see the paper
\cite[Lemma 3.2]{Dup}.

A special case
coincides with a recent expression obtained by Defant \cite[eq (29)]{Def}.
Using the polynomial representation, we  give simple proofs to a number of known identities for \(a(n)\),
including a convolution identity
found by Coons \cite{Coons} (Corollary \ref{Coon}),
and we derive a result on the divisibility of \(a(n)\) (Corollary \ref{nc}).

Let $c_i\geq 0$ be non-negative integers, and define $d_i=\sum_{j=1}^i c_i$,
$[c_1,\cdots,c_r]=2^{d_r}+2^{d_{r-1}}+\cdots+2^{d_1}+1$.
Note that if $c_i\geq 1$, then $c_i$ is the distance between the
 two consecutive 1's corresponding to $2^{d_{i-1}}$ and $2^{d_i}$
in the binary expansion of the odd integer
$n=[c_1,\cdots,c_r]$. In this case, clearly $r+1=s(n)$, the sum of the digits
of the binary expansion of $n$. However, in general
$2^{d_r}+2^{d_{r-1}}+\cdots+2^{d_1}+1$ is
not necessarily the binary expansion of $n$. If $r<i$, we define $[c_i,\ldots, c_r]=1$.

Formula (\ref{ac}) below is proved by induction on \(c\), and the others  follow easily from the definitions.
\begin{eqnarray}
a(2^cn+1)&=&a(n)c+a(n+1) \label{ac},\\ \
[c_1,\ldots,c_r]& =& 2[c_1-1,c_2,\ldots,c_r]-1,\ \ c_1>0\\   \
[c_1,\ldots,c_r]&=&1+2^{c_1}[c_2,\ldots,c_r]  \label{2c}\\
a([c_1,\ldots,c_r]-1)&=&a([c_2,\ldots,c_r]) \label{c2}\\
a([c_1,\ldots,c_r]+1)&=&a([c_1-1,c_2,\ldots,c_r]), \ \ c_1>0 \label{cm1}
\end{eqnarray}

\begin{lemma} \label{l2}
Let $c_i$ be a sequence of non-negative integers, and suppose that $c_2>0$.
Then, for each $r\geq 2$,
\begin{equation} \label{l2eq}
a([c_1,c_2,\ldots,c_r])=(c_1+1)a([c_2,c_3,\ldots,c_r])-a([c_3,\ldots,c_r]).
\end{equation}
\end{lemma}
\begin{proof}
Using (\ref{2c}), (\ref{cm1}), and  (\ref{ac}) with \(c=c_1\), \(n=[c_2,\ldots,c_r]\), we find
\begin{equation} \label{sh1}
a([c_1,\ldots,c_r])=c_1a([c_2,\ldots,c_r])+a([c_2-1,c_3,\ldots,c_r]).\end{equation}
Then using (\ref{c2}) and (\ref{ac}) again with \(n=[c_2-1,c_3,\ldots,c_r]\), we find
\begin{equation}\label{sh2}
a([c_2,\ldots,c_r])=a([c_2-1,c_3,\ldots,c_r]+a([c_3,\ldots,c_r]).\end{equation}
Comparing (\ref{sh1}) and (\ref{sh2}), the result follows.
\end{proof}

\begin{theorem} \label{th}
Let $c_i: 1\leq i \leq m$ be a sequence of positive integers. Then
\[a([c_1,c_2,\ldots, c_m])=q_r(c_1+1,c_2+1,\ldots, c_m+1).\]
\end{theorem}
\begin{proof}
Define polynomials $p_r$ by $p_0=1$, and
$p_r(x_1,\ldots,x_r)= q_r(x_1+1,\ldots,x_r+1).$
By Lemma (\ref{l1}) and Lemma (\ref{l2}), both $p_r(c_1,\ldots,c_r)$
and $a([c_1,\ldots,c_r])$ satisfy the polynomial recurrence relation
\begin{equation}\label{0}
T_r(x_1,\ldots,x_r)=(x_1+1)T_{r-1}(x_2,\ldots,x_r)-T_{r-2}(x_3,\ldots,x_r),
\end{equation}
with initial conditions
$T_0=1$, $T_1(x_1)=x_1+1.$
\end{proof}

The expressions for \(a([c_1,\ldots,c_r)]\) for the first few values of \(r\)  are recorded below. The case \(r=3\) was recently
 derived
by Defant in \cite[eq (29)]{Def}.
\[\begin{aligned}
a([c_1])&=c_1+1\\
a([c_1,c_2])&=c_1+c_1c_2+c_2\\
a([c_1,c_2,c_3])&=c_1 c_2 c_3+ c_1 c_2 + c_1 c_3 + c_2 c_3+c_2 -1 \\
a([c_1,c_2,c_3,c_4])&=c_1 c_2 c_3 c_4  + c_1 c_2 c_3 + c_2 c_3 c_4 + c_1 c_3 c_4\\
    &+ c_1 c_2 c_4+ c_2 c_4+ c_1 c_3+ c_2 c_3-c_1-c_4 -1\\
\end{aligned}
\]

The following corollary of the previous theorem corresponds to Corollary 3.3 of \cite{Dup}.

\begin{corollary} \label{cor1}
For \(r\geq 1\), define the matrix
\[M_r(y_1,\ldots,y_r)=\begin{pmatrix}
y_1 & 1 & 0 & 0 & 0 \\
1 & y_2 & 1 & 0 & 0 \\
0 & \ddots & \ddots & \ddots & 0\\
0 & 0 & \ddots & \ddots & 1\\
0 & 0 & 0 & 1 & y_r
\end{pmatrix}
\]
Let \(c_i: 1\leq i\leq m \) be positive integers, and let \(I_r\) be the \(r\times r\) identity matrix. Then
\[a([c_1,\ldots,c_r])=\det(I_r+M_r(c_1,\ldots,c_r)).\]
\end{corollary}
\begin{proof}
It is easy to check that \(\det (M_r(y_1,\ldots,y_r))\) satisfies the same recurrence relation as \(q_r(y_1,\ldots,y_r)\), with the same
initial conditions. So   \(\det (M_r(y_1,\ldots,y_r))=q_r(y_1,\ldots,y_r)\).
\end{proof}

While no easily discernible pattern is apparent in the polynomials \(p_r\)
(or equivalently the expressions for \(a([c_1,\ldots,c_r])\) given above),
listing the first few polynomials \(q_r\) reveals a surprising structure:

\[\begin{aligned}
q_2(y_1,y_2)&=y_1y_2-1\\
q_3(y_1,y_2,y_3)&=y_1y_2y_3-y_1-y_3\\
q_4(y_1,y_2,y_3,y_4)&=y_1y_2y_3y_4-y_1y_2-y_1y_4-y_3y_4 \\
q_5(y_1,y_2,y_3,y_4,y_5)&=y_1y_2y_3y_4y_5-y_1 y_2 y_3-y_1 y_2 y_5-y_1 y_4 y_5-y_3 y_4 y_5+y_1+y_3+y_5
\end{aligned}
\]

The next theorem gives a precise description of this structure.
For integers $r\geq 1$ and $1\leq s \leq r$, define the sets
\[A_{r,s}=\{(i_1,i_2,\ldots,i_s): 1\leq i_1 <i_2<\cdots < i_s\leq r, i_j\equiv
j \pmod{2}\},\]
and $A_{r,0}=\{0\}$.
So $A_{r,s}$ consists of increasing sequences of integers that start with an
odd number and then alternate between even and odd numbers. For example,
\(A_{5,3}=\{(1,2,3),(1,2,5),(1,4,5),(3,4,5)\}\).

If $u=(i_1,i_2,\ldots,i_s)\in A_{r,s}$, we write $y_u=y_{i_1}y_{i_2}\cdots y_{i_s}$, and $y_0=1$.
 For $r\geq 1$ and $0\leq s \leq r$,
let \( \omega_{r,s}=(-1)^r\cos\left(\pi(r+s)/2\right)\).
\begin{theorem} \label{l1}
If $r\geq 2$, then
\begin{equation}\label{Ars}
q_r(y_1,y_2,\ldots,y_r)=\sum_{s=0}^r \omega_{r,s}\sum_{u\in A_{r,s}}y_u.\end{equation}
\end{theorem}
\begin{proof}
We will show that the right side of (\ref{Ars}) satisfies the recurrence (\ref{recur}).
If $r\geq 1$ and $1\leq s \leq r$, let
\[B_{r,s}=\{(i_1,i_2,\ldots,i_s)\in A_{r,s}: i_1=1\},\]
\[C_{r,s}=\{(i_1,i_2,\ldots,i_s)\in A_{r,s}: i_1>1\}.\]
Then $A_{r,s}$ is the disjoint union of $B_{r,s}$ and $C_{r,s}$. There are
bijections $
\phi:  B_{r,s}\rightarrow A_{r-1,s-1}$ and $
\psi: C_{r,s} \rightarrow A_{r-2,s}$ given by
$\phi\left((1,i_2,\ldots, i_s)\right)=(i_2-1,i_3-1,\ldots,i_s-1)$ and
$\psi\left( (i_1,i_2,\ldots, i_s)\right) =(i_1-2, \ldots, i_s-2)$.
If $z_i=y_{i+1}$, $1\leq i \leq r-1$ and $w_i=y_{i+2}$, $1\leq i \leq r-2$, then
$y_u=y_1z_{\phi(u)}$ for $u\in B_{r,s}$ and $y_u=w_{\psi(u)}$
for $u\in C_{r,s}$. The result then easily follows by splitting the sum over \(A_{r,s}\) as
\(\sum_{u\in B_{r,s}}+\sum_{u\in C_{r,s}}\), and using the fact that
\(\omega_{r,s+1}=\omega_{r-1,s}\), and \(\omega_{r,s}=-\omega_{r-2,s}\).
\end{proof}


The following corollary is attributed to B. Reznick in \cite{Sam} (see \cite[Lemma 2.5]{Rez}).
\begin{corollary} \label{Cor}
If $n$ is a positive integer, let $\overline{n}$ denote the integer obtained by reading the digits in the binary expansion of $n$ in reverse order. Then
$a(\overline{n})=a(n)$.
\end{corollary}
\begin{proof}
It is enough to consider the case $n$ odd. Note that if $n=[c_1,\ldots,c_m]$,
then $\overline{n}=[c_m,\ldots, c_1]$. So the result will follow if we show that
$q_m(y_1,\ldots,y_m)=q_m(y_m,\ldots,y_1)$. This follows easily from either Corollary \ref{cor1},
by a permutation of the rows and columns that reverses the main diagonal of the matrix \(M_r\), or from
 Theorem \ref{l1}, by noticing that there is an involution $\beta_{r,s}$ on the sets
$A_{r,s}$ given by $(i_1,\ldots,i_s)\mapsto (i'_1,\ldots,i'_s)$, where
$i'_j=r-i_{s-j+1}+1$, because if $r$ and $s$ have the same parity, then
$r-i_{s-j+1}+1\equiv r-s+j\equiv j \pmod{2}$, while if $r\not \equiv s \pmod{2}$, then $\omega_{r,s}=0$.
\end{proof}

\begin{proposition}\label{prop}
If \(r\geq 0, k\geq 0\), then
\begin{eqnarray*} q_{k+r}(t_1,\ldots,t_k,y_1,\ldots,y_r)&=&q_k(t_1,\ldots,t_k)q_r(y_1,\ldots, y_r) \\
&-& q_{k-1}(t_1,\ldots,t_{k-1})q_{r-1}(y_2,\ldots,q_r).\end{eqnarray*}
\end{proposition}
\begin{proof}
The proposition is proved by induction on \(k\), by writing \(q_{r+(k+1)}=q_{(r+1)+k}\), and making use of
Corollary \ref{Cor}.
\end{proof}

As a consequence of the last proposition, we obtain a simple proof of the following result of Coons \cite{Coons}.
\begin{corollary} \label{Coon}
If \(e\), \(u\) and \(c\) are non-negative integers with \(c\leq 2^e\), then
\[a(c)a(2u+5)+a(2^e-c)a(2u+3)=a(2^e(u+2)+c)+a(2^e(u+1)+c).\]
\end{corollary}
\begin{proof}
The result holds for \(c=0\) trivially,  and for \(c=1\) or \(u=0\) by using the basic identities for the Stern's sequence.  Decreasing \(e\) if necessary, we may assume that \(c\) is odd and \(c\geq 3\). So there is some \(k\geq 2\) and integers \(c_1,\ldots,c_{k-1}\) such that
 \(c=[c_1,\ldots,c_{k-1}]\). Since \(c\leq 2^e\), we can define the positive integer \(c_k=e-(c_1+\cdots+c_{k-1})\), and then
 \([c_1,\ldots,c_k]=c+2^e\).
Write \(u+2=[u_1,\ldots,u_r]\) for some positive integers \(u_1,\ldots,u_r\).  Then
\([c_1,\cdots,c_k,u_1,\cdots,u_r]=c+2^e (u+2)\), and it is easily checked that
 \(q_{r-1}(u_2+1,\ldots,u_r+1)=a(u+1)\).
 Proposition \ref{prop},
(with \(y_i=u_i+1, t_i=c_i+1\), gives us the identity
\[a(c+2^e(u+2))+a(c)a(u+2) = a(c+2^e)a(u+2)+a(c)a(u+3).\]
Use  \(a(2u+3)=a(u+2)+a(u+1)\), \(a(2u+5)=a(u+2)+a(u+3)\), the basic identity
\(a(2^e+c)=a(2^e-c)+a(c)\) (see\cite{Sam}) and the identity \(a(2^e-c)a(u+1) +a(c)a(u+2)=a(2^e(u+1)+c)\) (easily proved by induction on
 \(e\)) to get the result.
 \end{proof}

The following result is an easy consequence of the
recurrence (\ref{recur}) satisfied by the polynomials \(q_r\).
Recall that $s(n)$ is the number of 1's appearing in the binary expansion
of $n$.

\begin{corollary} \label{nc}
Suppose $k$ is a positive integer that divides
 the exponent of each power of 2 appearing in the binary expansion of $n$.
 Then:
 \[a(n)\equiv
 \begin{cases} 0 \pmod{k} &\mbox{if } s(n)\equiv 0 \mbox{ or } 3  \pmod{6} \\
1 \pmod{k} & \mbox{if } s(n) \equiv 1 \mbox{ or } 2 \pmod{6}\\
-1 \pmod{k} & \mbox{if } s(n) \equiv 4 \mbox{ or } 5 \pmod{6}. \end{cases}
\]
\end{corollary}
\begin{proof}
We may assume that $n$ is odd. Let $m=s(n)$.
By assumption, $n=[kc_1,\ldots,kc_m]$, and by Theorem \ref{th}, $a(n)=q_m(kc_1+1,\ldots,kc_m+1)\equiv q_m(1,\ldots,1) \pmod{k}$.
If $b_m=q_m(1,\ldots,1)$, then $b_m$ satisfies the recurrence
$b_m=b_{m-1}-b_{m-2}$ with initial conditions $b_0=b_1=1$. This
 recurrence is easily solved as $b_m=\cos(m\pi/3)+\sin(m \pi/3)/\sqrt{3}$ and
 the result follows.
\end{proof}

We conclude with a Binet type formula easily obtained by letting all variables of \(q_r\) equal a single variable \(t\).
\begin{corollary}
For all integers $r\geq 1$ and $t\geq 2$,
\[a\left(\frac{2^{rt}-1}{2^t-1}\right)=\frac{\lambda^r-\mu^r}{\sqrt{(t-1)(t+3)}},\]
where
\[\lambda=\frac{t+1+\sqrt{(t-1)(t+3)}}{2}, \ \ \ \mu=\frac{t+1-\sqrt{(t-1)(t+3)}}{2}.\]
\end{corollary}
\begin{proof}
Note that $(2^{rt}-1)/(2^t-1)=[t,\ldots,t]$, where there are $r-1$ entries. So $b_r= a((2^{rt}-1)/(2^t-1))=q_{r-1}(t+1,\ldots,t+1)$ satisfies the recurrence $b_r=(t+1)b_{r-1}-b_{r-2}$ with initial conditions $b_1=1$, $b_2=t+1$. Solving this recurrence we obtain the result.
\end{proof}\\

\noindent
{\bf Remark}
The previous corollary lends itself to natural generalizations, by considering for example \(q_r(t,s,t,s,\ldots)\) or
\(q_r(t,s,u,t,s,u,\ldots)\) and so on. We leave the exploration of the corresponding formulas for the Stern's sequence
to the interested reader. \\

\noindent
{\bf Acknowledgment}
I thank Sam Northshield for pointing out the article \cite{Dup} (that
 led to the current title of this note) and for several helpful comments
on an earlier version of the paper, and Christophe Vignat for pointing out the article \cite{Coons}.

\end{document}